\documentclass[a4paper, 12pt]{article}

\usepackage{amsfonts, amsmath, amsthm}
\usepackage{amssymb}
\usepackage{latexsym}
\usepackage[dvips]{graphicx}
\usepackage{color}

\theoremstyle{plain}
\newtheorem{thm}{Theorem}[section]
\newtheorem{prp}{Proposition}[section]
\newtheorem{lem}{Lemma}[section]

\theoremstyle{definition}

\theoremstyle{remark}
\newtheorem{rmk}{Remark}[section]

\numberwithin{equation}{section}

\newcommand{\Z}{\mathbb{Z}}

\newcommand{\R}{\mathbb{R}}

\newcommand{\Sph}{\mathbb{S}}

\newcommand{\pa}{\partial}
\newcommand{\eps}{\varepsilon}
\newcommand{\jb}[1]{\langle #1 \rangle}

\DeclareMathOperator{\supp}{\rm supp}

\begin{document}
\title{
 On Agemi-type structural conditions for a system of semilinear wave 
equations}

\author{
          Yoshinori Nishii\thanks{
              Department of Mathematics, Graduate School of Science, 
              Osaka University. 
              1-1 Machikaneyama-cho, Toyonaka, Osaka 560-0043, Japan. 
              (E-mail: {\tt y-nishii@cr.math.sci.osaka-u.ac.jp})             }
           \and  
          Hideaki Sunagawa \thanks{
              Department of Mathematics, Graduate School of Science, 
              Osaka City University. 
              3-3-138 Sugimoto, Sumiyoshi-ku, Osaka 558-8585, Japan. 
              (E-mail: {\tt sunagawa@sci.osaka-cu.ac.jp})
             }
}

\date{\today }   
\maketitle

\noindent{\bf Abstract:}\ 
We consider a two-component system of cubic semilinear wave equations 
in two space dimensions satisfying the Agemi-type structural condition (Ag) 
but violating (Ag$_0$) and (Ag$_+$). 
For this system, we show that small amplitude solutions are asymptotically 
free as $t\to +\infty$. 
\\

\\


\section{Introduction }  \label{sec_intro}

This paper is devoted to the study on large-time asymptotic behavior of 
solutions $u=(u_1,u_2)$ to 
\begin{align}
\left\{
\begin{array}{l}
 \Box u_1=- (\pa_t u_2)^2 \pa_t u_1,\\
 \Box u_2=- (\pa_t u_1)^2 \pa_t u_2,
 \end{array}\right.
 &\quad (t,x )\in (0,\infty)\times \R^2
\label{eq}
\end{align}
with the initial condition 
\begin{align}
u_j(0,x)=\eps f_j(x),\ \ 
\pa_t u_j(0,x)=\eps g_j(x), 
\qquad  x \in \R^2,\ j=1,2,
\label{data}
\end{align}
where $\eps>0$ is a small parameter, $\Box=\pa_t^2-\pa_{x_1}^2-\pa_{x_2}^2$, 
and $f_j$, $g_j\in C_0^{\infty}(\R^2)$. 

Before getting into the details, let us recall the backgrounds briefly 
to make clear why this system is of our interest. 
To put \eqref{eq} in perspective, let us first consider more general systems 
in the form 
\begin{align}
 \Box u=F(\pa u), 
  \qquad (t,x)\in (0,\infty)\times \R^d,
\label{eq_gene}
\end{align}
with $C_0^{\infty}$-data of size $\eps$, where $u=(u_j(t,x))_{1\le j\le N}$, 
$\pa_0=\pa/\pa t$, $\pa_k=\pa/\pa x_k$ ($1\le k\le d$), 
$\Delta=\pa_{1}^2+\cdots +\pa_{d}^2$, $\Box=\pa_t^2-\Delta$ and 
$\pa u=(\pa_{a} u_j)_{0\le a \le d; 1\le j\le N}$. 
$F=(F_j)_{1\le j\le N}$ is an $\R^N$-valued $C^{\infty}$-function vanishing 
of order $p \ge 2$ in a neighborhood of $0\in \R^{N\times (1+d)}$. 
If $p>1+2/(d-1)$ and $\eps$ is small enough, it is well-konwn that 
\eqref{eq_gene} admits a unique global $C^{\infty}$-solution and it behaves 
like a solution to the free wave equation as $t\to \infty$, 
while if $p\le 1+2/(d-1)$, global existence fails to hold in general 
even when $\eps>0$ is arbitrarily small (\cite{J0}, \cite{Go}, etc). 
In this sense, the power $p_{\rm c}(d):=1+2/(d-1)$ is a critical exponent 
for nonlinear perturbation. Note that $p_{\rm c}(2)=3$ and $p_{\rm c}(3)=2$. 
On the other hand, the small data global existence can hold for some class 
of nonlinearity of the critical power. One of the most successful example 
is the so called {\em null condition}, which has been originally 
introduced by 
Christodoulou\cite{Chr} and Klainerman\cite{Kla} in three dimensional case 
and developed later by several authors (see \cite{Go}, \cite{Hos}, 
\cite{K0}, \cite{Ali} etc., for the two-dimensional counterparts). 
We remark that the global solution $u$ under the null condition is 
asymptotically free in the sense that there exists a solution $u^+$ to 
the free wave equation $\Box u^+=0$ such that 
\[
 \lim_{t\to \infty}\|u(t)-u^+(t)\|_E=0,
\]
where the energy norm $\|\, \cdot\, \|_{E}$ is defined by 
\[
 \|\phi(t)\|_{E}^2
=\frac{1}{2} \int_{\R^d} \sum_{a=0}^{d}|\pa_a \phi(t,x)|^2\, dx.
\]
When we restrict to the case where $d=2$ and 
the nonlinearity is given by 
\begin{align}\label{F_explicit}
 F_j(\pa u)
 =  \sum_{k,l,m=1}^{N}\sum_{a,b,c=0}^{2}  C_{jklm}^{abc} 
 (\pa_a u_k)(\pa_b u_l)(\pa_c u_m)
\end{align}
with real constants $C_{jklm}^{abc}$, 
the null condition is satisfied if and only if 
$F_j^{\rm red}(\omega, Y)$ vanishes identically on $\Sph^1\times \R^N$, 
where 
\[
 F_j^{\rm red}(\omega, Y)
 =  \sum_{k,l,m=1}^{N}\sum_{a,b,c=0}^{2}  C_{jklm}^{abc} 
 \omega_a \omega_b \omega_c Y_k Y_l Y_m
\]
for $Y=(Y_j)_{1\le j \le N} \in \R^N$ 
and 
$\omega=(\omega_1,\omega_2) \in \Sph^1$, with the convention $\omega_0=-1$. 

Recently, a lot of efforts have been made for the study on weaker structural 
conditions than the null condition mentioned above 
which ensure the small data global existence
(see e.g., \cite{Lind1}, \cite{LR1}, \cite{LR2}, \cite{Ali2}, \cite{Ali3}, 
\cite{Hos2}, \cite{Kub}, \cite{KK}, \cite{KMuS}, \cite{KMatoS}, \cite{KMatsS}, 
\cite{K2}, \cite{K3}, \cite{HidYok}, etc). 
It should be emphasized that the situation becomes much more complicated 
because long-range nonlinear effects must be taken into account.
In \cite{KMatsS}, the following condition has been introduced: 
\begin{enumerate}
\item[{\bf (Ag)}] \ 
There exists an $N\times N$-matrix valued continuous function 
${\mathcal A}={\mathcal A}(\omega)$ on $\Sph^1$, which is a 
positive-definite symmetric matrix for each $\omega\in \Sph^1$, such that
\[
 Y \cdot {\mathcal A}(\omega) F^{\rm red}(\omega, Y)\ge 0,
 \quad (\omega, Y)\in \Sph^1 \times \R^N,
\]
where the symbol $\,\cdot\,$ denotes the standard inner product in 
$\R^N$.
\end{enumerate}
After the partial results \cite{Kub}, \cite{Hos2}, \cite{KMuS}, 
it has been shown in \cite{KMatsS} that (Ag) implies the small data global 
existence for \eqref{eq_gene}--\eqref{F_explicit} in two space dimensions 
(see also \cite{KimS}, \cite{Kim}, \cite{LiS} etc., 
for closely related works). 
We note that this condition is motivated by works of Rentaro Agemi in 
the late 1990's. He tried to find a structural condition which covers 
not only the standard null condition but also the wave equations with 
cubic nonlinear damping such as $\Box v=-(\pa_t v)^3$. 
Therefore it would be fair to 
call this the {\em Agemi-type condition}. 
As for the asymptotic behavior of the global solutions under (Ag), 
many interesting problems seem left unsolved. 
To the authors' knowledge, only the following two cases (Ag$_+$) and (Ag$_0$) 
are well-understood: 

\begin{enumerate}
\item[{\bf(Ag$_+$)}] \ 
There exist an ${\mathcal A}(\omega)$ as in (Ag) and 
a positive constant $C$  such that 
\begin{align} 
Y \cdot {\mathcal A}(\omega) F^{\rm red}(\omega, Y)\ge C|Y|^4,
 \quad (\omega, Y)\in \Sph^1 \times \R^N.
\label{agemi_plus}
\end{align}
\end{enumerate}
Note that \eqref{agemi_plus} is equivalent to 
\[
Y \cdot {\mathcal A}(\omega) F^{\rm c, red}(\omega, Y)\ne 0,
 \quad (\omega, Y)\in \Sph^1 \times (\R^N\backslash\{0\}),
\]
if (Ag) is satisfied and $F$ is cubic. 
Under (Ag$_+$), the total energy $\|u(t)\|_{E}$ 
decays like $O((\log t)^{-1/4+\delta})$ as $t\to +\infty$, where
$\delta>0$  can be arbitrarily small. See \cite{KMatsS} for the detail. 

\begin{enumerate}
\item[{\bf (Ag$_0$)}] \ 
There exists an ${\mathcal A}(\omega)$ as in (Ag) 
such that 
\[ 
Y \cdot {\mathcal A}(\omega) F^{\rm red}(\omega, Y)= 0,
 \quad (\omega, Y)\in \Sph^1 \times \R^N.
\]
\end{enumerate}
Note that (Ag$_0$) is stronger than (Ag) if $F$ is cubic 
(while it is equivalent to (Ag) in the quadratic case). 
Roughly speaking, it holds under (Ag$_0$) that 
\[
 \pa u(t,x) \sim |x|^{-1/2}\hat{\omega}(x) V(t;|x|-t,x/|x|)
\]
as $t\to \infty$, where $\hat{\omega}(x)=(-1, x_1/|x|, x_2/|x|)$, and 
$V(t;\sigma,\omega)$ solves 
\[
 \pa_t V=\frac{1}{t} Q(\omega,V)V
\]
with a suitable skew-symmetric matrix $Q$ depending on $(\omega, V)$. 
In particular, decay of the total energy never occurs under (Ag$_0$) 
except for the trivial solution. Typical example satisfying (Ag$_0$) is
 \begin{align*}
\left\{
\begin{array}{l}
 \Box u_1=-  (\pa_t u_1)^2 \pa_t u_2,\\
 \Box u_2= (\pa_t u_1)^3.
 \end{array}\right.
\end{align*}
For more details on (Ag$_0$), 
see \cite{KMatoS}, \cite{K3} and Chapter 10 in \cite{K2}.

Now, let us turn back to our system \eqref{eq}, that is the case where 
$F_1(\pa u)=-(\pa_t u_2)^2 \pa_t u_1$, 
$F_2(\pa u)=-(\pa_t u_1)^2 \pa_t u_2$ and $N=2$ in \eqref{eq_gene}. 
We can easily check that (Ag) is satified by \eqref{eq} with 
${\mathcal A}(\omega)$ being the $2\times 2$ identity matrix. Indeed 
we have 
$Y \cdot F^{\rm red}(\omega, Y)=2Y_1^2Y_2^2$. 
Note also that both (Ag$_+$) and (Ag$_0$) are violated. 
We observe that the system \eqref{eq} possesses two conservation laws 
\begin{align}\label{CL1}
 \frac{d}{dt}\left(
 \|u_1(t)\|_{E}^2+\|u_2(t)\|_{E}^2\right)
=
- 2\int_{\R^2} \bigl(\pa_t u_1(t,x) \bigr)^2 \bigl(\pa_t u_2(t,x) \bigr)^2\, dx
\end{align}
and
\begin{align}\label{CL2}
 \frac{d}{dt}\Bigl( \|u_1(t)\|_{E}^2 - \|u_2(t)\|_{E}^2\Bigr)=0.
\end{align}
However, these are not enough to say something about the large-time 
asymptotics for $u(t)$, and this is not trivial at all. 
To the authors' knowledge, 
there are no previous results which cover the asymptotic behavior of 
solutions to \eqref{eq}--\eqref{data}. 
The aim of the present paper is to address this point. 
Several related issues will be discussed elsewhere. 

The main result is as follows.

\begin{thm}\label{thm_main}
Suppose that $f$, $g\in C_0^{\infty}(\R^2)$ and $\eps$ is suitably small. 
Then the global solution $u(t)$ to \eqref{eq}--\eqref{data} is asymptotically 
free in the following sense: 
there exists $(f^+,g^+) \in \dot{H}^1(\R^2)\times L^2(\R^2)$ such that 
\[
 \lim_{t\to +\infty} \|u(t)-u^+(t)\|_E=0,
\]
where $u^+$ solves the free wave equation $\Box u^+=0$ with 
$(u,\pa_t u)|_{t=0}=(f^+,g^+)$.
\end{thm}

\begin{rmk} If $(f_1,g_1)=(f_2,g_2)$, then the system is reduced to 
the single euation $\Box v=-(\pa_t v)^3$. 
Therefore we can adapt the result of \cite{KMuS}, \cite{KMatsS} to see that 
the total energy $\|u(t)\|_{E}$ decays like $O((\log t)^{-1/4+\delta})$
as $t\to \infty$. On the 
other hand, if $\|u_1(0)\|_{E}\ne \|u_2(0)\|_{E}$, then 
at least one component $u_1$ or $u_2$ tends to a non-trivial free solution 
because of the conservation law \eqref{CL2}.
\end{rmk}

\begin{rmk}
Our proof of Theorem \ref{thm_main} does not rely on the conservation laws 
\eqref{CL1} and \eqref{CL2} at all. For example, the same proof is valid for 
the system
\begin{align*}
\left\{
\begin{array}{l}
 \Box u_1=- |\nabla_x u_2|^2 \pa_t u_1,\\
 \Box u_2=- |\nabla_x u_1|^2 \pa_t u_2,
 \end{array}\right.
\end{align*}
or more generally, any cubic terms satisfying 
the standard null condition can be added to the right-hand side of it.
\end{rmk}

\begin{rmk}The above theorem concerns only the forward Cauchy problem 
(i.e., for $t>0$). 
For the backward Cauchy problem, it is not difficult to 
construct a blowing-up solution (with a suitable choice of $f$, $g$) 
based on the idea of \cite{Go}. 
This should be contrasted with the behavior of solutions under (Ag$_0$). 
\end{rmk}

\section{Preliminaries}  \label{sec_prelin}

In this section, we collect several notations which will be used in the 
subsequent sections. 
For $z\in \R^d$, we write $\jb{z} =\sqrt{1+|z|^{2}}$.
We define
\begin{align*}
 S := t \pa_{t} + x_{1}\pa_{1} + x_{2}\pa_{2}, \ 
 L_{1} := t \pa_{1} + x_{1} \pa_{t}, \ 
 L_{2} := t \pa_{2} + x_{2} \pa_{t}, \ 
 \Omega := x_1 \pa_2 - x_2 \pa _1,
\end{align*}
and we set
$\Gamma =(\Gamma_{j})_{0 \le j \le 6}
=(S , L_{1} , L_{2} , \Omega , \pa_{0} , \pa_{1} , \pa_{2})$.
For a multi-index 
$\alpha = (\alpha _{0},\alpha _{1},\cdots,\alpha _{6}) \in \Z^{7}_{+}$, 
we write 
$|\alpha |=\alpha _{0}+\alpha _{1}+\cdots+\alpha _{6}$ 
and 
$\Gamma^{\alpha}
=\Gamma_{0}^{\alpha_{0}}\Gamma_{1}^{\alpha_{1}}\cdots \Gamma_{6}^{\alpha_{6}}$, where $\Z_{+}=\{n\in\Z ; n\ge 0\}$.
We define $|\, \cdot\,|_s$ by 
\begin{align*}
 |\phi (t,x)|_{s}=\sum _{|\alpha |\le s}|\Gamma^{\alpha }\phi (t,x)|.
\end{align*}
For $x\in \R^2\backslash\{0\}$, we write 
$r:=|x|$, 
$\omega =(\omega_1,\omega_2):=x/|x|$, 
$\omega^{\perp}=(\omega_1^{\perp},\omega_2^{\perp}):=(-\omega_2,\omega_1)$, 
$\partial _{r}:=\omega _{1}\partial _{1}+\omega _{2}\partial _{2}$, 
and $\pa_{\pm}:=\partial _{t}\pm \partial _{r}$.
Important relations are 
\begin{align}\label{dal_polar}
 \pa_{+}\pa_{-}(r^{1/2}\phi)
 =
 r^{1/2}\Box \phi+\frac{1}{4r^{3/2}}(4\Omega^{2}+1)\phi,
\end{align}
\begin{align}\label{dx_polar}
(t+r)(\pa_j-\omega_j\pa_r)= 
\omega_j^{\perp}(\Omega + \omega_1 L_2 - \omega_2 L_1), 
\qquad j=1,2,
\end{align}
\begin{align}\label{d_plus}
(t+r)\pa_+=S+\omega_1L_1+\omega_2L_2,
\end{align}
and $\pa_+ + \pa_-=2\pa_t$, $\pa_+ - \pa_-=2\pa_r$. 
Next we set 
$\Lambda_{\infty}=\{ (t,x) \in [0,\infty)\times \R^2 ; |x|\geq t/2\geq 1 \}$ 
and $\mathcal{D}=-2^{-1}\pa_{-}$.
Then we have the following. 

\begin{lem}\label{lem1}
There exists a positive constant $C$ such that
\begin{align*}
\left| 
|x|^{1/2}\pa\phi(t,x)-\hat \omega(x) \mathcal{D}\left( |x|^{1/2}\phi(t,x) \right) 
\right|
\le C \jb{t+|x|}^{-1/2}| \phi(t,x) |_1
\end{align*}
for $(t,x)\in \Lambda_\infty$, where 
$\hat{\omega}(x)=(-1,x_1/|x|, x_2/|x|)$.
\end{lem}
This is a consequence of \eqref{dx_polar} and \eqref{d_plus}. 
See Corollary~3.3 in \cite{KMuS} for more detail of the proof.

\section{The John--H\"ormander reduction}  \label{sec_prof_eq}

In this section, we will make reductions of the problem along the approach exploited in \cite{KMuS}, \cite{KMatoS}, \cite{KMatsS}, \cite{K3}. 
The essential idea goes back to John\cite{J} and H\"ormander \cite{Hor} 
concerning detailed lifespan estimates for quadratic quasilinear wave 
equations in three space dimensions.

Let $u=(u_1,u_2)$ be a smooth solution to \eqref{eq}--\eqref{data} on 
$[0,\infty) \times \R^2$. Since $f$ and $g$ are compactly-supported, we can 
take $R>0$ such that 
$\supp f \cup \ \supp g \subset \{ x \in \R^2 ; |x| \le R \}$.
Then, by the finite propagation property, we have
\begin{align}
 \supp u(t,\cdot) \subset \{ x \in \R^2 ; |x| \le t+R \}
\label{finite_propa}
\end{align}
for $t\ge 0$. 
We define $U=(U_1,U_2)$ by
$U_j(t,x)=\mathcal{D}(|x|^{1/2}u_j(t,x))$, $j=1,2$. 
We also introduce $H=(H_1,H_2)$ by
\[
\begin{array}{l}
H_1
=
\dfrac{1}{2}
\left( r^{1/2}(\pa_t u_2)^2(\pa_t u_1)+\dfrac{1}{t}U_2^2U_1 \right) 
-\dfrac{1}{8r^{3/2}}(4\Omega^2+1)u_1, \\[3mm]
H_2
=
\dfrac{1}{2}
\left( r^{1/2}(\pa_t u_1)^2(\pa_t u_2)+\dfrac{1}{t}U_1^2U_2 \right) 
-\dfrac{1}{8r^{3/2}}(4\Omega^2+1)u_2.
\end{array}
\]
By \eqref{dal_polar}, we have 
\begin{align}\label{profile}
\left\{\begin{array}{l}
\pa_{+} U_1(t,x)=\dfrac{-1}{2t} U_1(t,x) U_2(t,x)^2  + H_1(t,x), \\[3mm]
\pa_{+} U_2(t,x)=\dfrac{-1}{2t} U_1(t,x)^2 U_2(t,x) + H_2 (t,x).
\end{array}\right.
\end{align}
The following lemma tells us that $H$ can be regarded as a remainder 
if we have a good control of $u$ near the light cone.
\begin{lem}
There exists a positive constant $C$ which may depend on $R$ such that
\begin{align}
 |H(t,x)|
\le 
C t^{-1/2}\bigl(|\pa u| +\jb{t+|x|}^{-1}|u|_{1}\bigr)^2 |u|_1
+C t^{-3/2}|u|_2
\label{est_H}
\end{align}
for 
$(t,x)\in\Lambda_{\infty,R}:=
\{(t,x)\in \Lambda_{\infty}\,;\, |x|\le t+R \}$.
\label{remainder}
\end{lem}
For the proof, see Lemma 2.8 in \cite{KMatsS}.

Next we recall the  basic decay estimates satisfied by 
the global small amplitude solution $u$ to \eqref{eq}--\eqref{data}.
From the argument of Section 3 in \cite{KMatsS}, we already know the 
following.
\begin{lem}
Let $k \geq 4$, $0<\mu <1/10$ and $0<(8k+7)\nu <\mu$. 
Suppose that $\eps$ is suitably small. 
Then the solution $u$ to \eqref{eq}--\eqref{data} satisfies
\begin{align}
|u(t,x)|_{k+1} \le C \eps \jb{t+|x|}^{-1/2+\mu},
\label{apriori2}
\end{align}
\begin{align*}
|\pa u(t,x)|_k \le C \eps \jb{t+|x|}^{-1/2+\nu }\jb{t-|x|}^{\mu-1}
\end{align*}
and
\begin{align}
|\pa u(t,x)| \le C \eps \jb{t+|x|}^{-1/2}\jb{t-|x|}^{\mu-1}
\label{apriori1}
\end{align}
for $(t,x) \in [0,\infty)\times \R^2$, where $C$ is a positive constant 
independent of $\eps$.
\end{lem}

In what follows, we denote various positive constants by the same letter $C$ 
which may vary from one line to another. 
From \eqref{apriori2}, \eqref{apriori1}, \eqref{est_H} and 
Lemma~\ref{lem1}, we have
\begin{align}
|U(t,x)|
\le 
 \left| |x|^{1/2}\pa u(t,x) \right|
 + 
 \left| |x|^{1/2}\pa u(t,x)-\hat \omega U(t,x) \right|
\le 
C\eps \jb{t-|x|}^{\mu-1}
\label{est_U}
\end{align}
and
\begin{align}
|H(t,x)|
&\le 
C\eps^{2}t^{-1/2}\jb{t+|x|}^{\mu-1}\jb{t-|x|}^{\mu-1} 
+ C\eps t^{-3/2} \jb{t+|x|}^{\mu-1/2} \notag \\
&\le 
C\eps t^{2\mu-3/2}\jb{t-|x|}^{-\mu-1/2}
\label{H}
\end{align} 
for $(t,x) \in \Lambda_{\infty,R}$. 
Remember that the weights $|x|^{-1}$, $t^{-1}$, $(1+t)^{-1}$, 
$\jb{t+|x|}^{-1}$ are equivalent to each other on $\Lambda_{\infty,R}$. 
Indeed we have 
\begin{align*}
\jb{t+|x|}^{-1}
\le 
|x|^{-1}\le2t^{-1}
\le
3(1+t)^{-1}
\le
3(R+2)\jb{t+|x|}^{-1}.
\end{align*}
Now we make the final reduction. We set
\begin{align*}
\Sigma=\{ (t,x)\in[0,\infty)\times\R^2 ; 
|x|\geq t/2=1\ \mbox{or} \ |x|= t/2 \geq 1 \}
\end{align*}
and $t_{0,\sigma}=\max\{ 2 , -2\sigma \}$. 
Then, since 
the half line $\{ (t,(t+\sigma)\omega)\ ;\ t\geq0 \}$ meets  $\Sigma$ at 
the point $(t_{0,\sigma}, (t_{0,\sigma}+\sigma)\omega)$ for each 
$(\sigma,\omega)\in\R\times\Sph^1$, 
we can see that
\begin{align*}
\Lambda_{\infty,R}=\bigcup_{(\sigma,\omega)
\in(-\infty,R]\times\Sph^1} 
\{ \ (t,(t+\sigma)\omega)\ ;\ t\geq t_{0,\sigma} \}.
\end{align*}
We also note that there exists a positive constant 
$c_0$ depending only on $R$ such that
\begin{align}
c_0^{-1}\jb{\sigma} \le t_{0,\sigma}\le c_0 \jb{\sigma}
\label{t_0}
\end{align}
for $\sigma\in (-\infty,R]$. We set 
$V_j(t;\sigma,\omega)=U_j(t,(t+\sigma)\omega)$ 
and 
$K_j(t;\sigma,\omega)=H_j(t,(t+\sigma)\omega)$
for $(t;\sigma,\omega)\in [t_{0,\sigma},\infty)\times \R\times \Sph^1$, 
$j=1,2$. Then we can rewrite \eqref{profile} as 
\begin{align}
\left\{\begin{array}{l}
 \pa_t V_1(t)
  =
  \dfrac{-1}{2t}V_1(t) V_2(t)^2 + K_1(t), \\[3mm]
 \pa_t V_2(t)
 =
 \dfrac{-1}{2t} V_1(t)^2 V_2(t) + K_2(t),
\end{array}\right.
\label{profileV}
\end{align} 
which we call the {\em profile equation}. 
It follows from \eqref{est_U} and \eqref{H} that
\begin{align}\label{V}
|V(t;\sigma,\omega)| 
\le 
C\eps \jb{\sigma}^{\mu-1} 
\end{align}
and
\begin{align}\label{K}
|K(t;\sigma,\omega)| 
\le 
C\eps \jb{\sigma}^{-\mu-1/2} t^{2\mu-3/2} 
\end{align}
for 
$(t,\sigma,\omega)\in [t_{0,\sigma},\infty)\times (-\infty, R]\times \Sph^1$.

At the end of this section, let us summarize what has been done so far. 
By Lemma~\ref{lem1} and \eqref{apriori2}, 
the leading part for $\pa u_j(t,x)$ as $t\to \infty$ could be given by 
$|x|^{-1/2} \hat{\omega}(x)V_j(t;|x|-t,x/|x|)$, and, 
in view of \eqref{profileV}--\eqref{K}, 
the evolution of $V=(V_1,V_2)$ could be characterized by the system 
\begin{align*}
 \pa_t V_1
  =
  \dfrac{-1}{2t}  V_1 V_2^2, \qquad 
 \pa_t V_2
 =
 \dfrac{-1}{2t} V_1^2V_2,
\end{align*} 
up to harmless remainder terms. 
Our strategy of the proof of Theorem~\ref{thm_main} 
consists of two steps: the first is to investigate the asymptotic behavior of 
$V(t;\sigma,\omega)$ as $t\to+\infty$, and the second is to convert it 
into that of $\pa u(t,x)$.
They will be carried out in Sections \ref{sec_asympt} and 
\ref{sec_proof_main}, respectively.

\section{Asymptotics of solutions to the profile equation}  
\label{sec_asympt}

In this section, we focus on large-time behavior of 
$V(t;\sigma,\omega)$ introduced in the previous section. 
The goal here is to show the following.

\begin{prp} \label{prp_asym_prof}
Let $V=(V_j(t;\sigma,\omega))_{j=1,2}$ be as above. 
There exists 
$V^{+}=(V_j^+(\sigma,\omega))_{j=1,2}\in L^2(\R\times\Sph^1)$ such that
\begin{align}
\lim_{t\to\infty}\int_{\R}\int_{\Sph^1}
\left| \chi_t(\sigma)V(t;\sigma,\omega)-V^{+}(\sigma,\omega) \right|^2\, 
dS_{\omega}d\sigma=0,
\label{L2convergence}
\end{align}
where $\chi_{t}:\R\to\R$ is a bump function satisfying 
 $\chi_t(\sigma)=1$ for $\sigma>-t$ and $\chi_t(\sigma)=0$ for $\sigma\le-t$.
\end{prp}

Before going into the proof, let us introduce two simple lemmas.

\begin{lem}
\label{lem_M}
Let $C_0>0$, $C_1\ge 0$, $p>1$, $q>1$ and $t_0\ge2$. 
Suppose that $\Phi(t)$ satisfies 
\begin{align*}
\frac{d\Phi}{dt}(t)
\le \frac{-C_0}{t}\left|\Phi(t)\right|^p + \frac{C_1}{t^{q}} 
\end{align*}
for $t\ge t_0$. Then we have
\begin{align*}
\Phi(t)\le \frac{C_2}{(\log t)^{p^{*}-1}}
\end{align*}
for $t\ge t_0$, where $p^{*}$ is the H\"older conjugate of $p$ (i.e., 
$1/p+1/p^{*}=1$), and
\begin{align*}
C_2=\frac{1}{\log 2}\left( (\log t_0)^{p^{*}}\Phi(t_0) 
  + C_1\int_{2}^{\infty}\frac{(\log \tau)^{p^{*}}}{\tau^{q}}d\tau \right) 
  + \left( \frac{p^{*}}{C_0 p} \right)^{p^{*}-1}.
\end{align*}
\end{lem}

For the proof, see Lemma 4.1 of \cite{KMatsS}. 

\begin{lem}\label{lem_ODE}
Let $t_0\ge 0$ be given. For 
$\lambda$, $Q \in C\cap L^1([t_0,\infty))$, 
assume that $y(t)$ satisfies
\begin{align*}
\frac{dy}{dt}(t)= \lambda(t)y(t) + Q(t)
\end{align*}
for $t\ge t_0$. Then we have 
\[
 |y(t)-y^+|
 \le 
 C_3 \int_t^{\infty} \left( |y^+||\lambda(\tau)|+|Q(\tau)| \right) \, d\tau
\]
for $t\ge t_0$, where 
\[
 C_3=\exp\left(\int_{t_0}^{\infty} |\lambda(\tau)|\, d\tau \right)
\]
and
\[
 y^+=
 y(t_0) e^{\int_{t_0}^{\infty} \lambda(\tau)\, d\tau}
 + 
 \int_{t_0}^{\infty} Q(s) e^{\int_{s}^{\infty} \lambda(\tau)\, d\tau}\, ds.
\]
\end{lem}
\begin{proof}Put
\[
 \Phi(t;s)=\exp\left(\int_{s}^{t} \lambda(\tau)\, d\tau \right)
\]
for $s$, $t\in [t_0,\infty]$. Then we see that 
\begin{align*}
 y(t)
 = \Phi(t;t_0) y(t_0) +\int_{t_0}^t \Phi(t;s)Q(s)\, ds
 = \Phi(t;\infty) y^+ -\int_{t}^{\infty} \Phi(t;s)Q(s)\, ds.
\end{align*}
We also note that $|\Phi(s;t)|\le C_3$ and that 
\[
 |\Phi(t;\infty)-1| \le C_3\int_t^{\infty} |\lambda(\tau)|\, d\tau.
\]
Therefore we obtain 
\begin{align*}
 |y(t)-y^+|
 &\le 
  |\Phi(t;\infty)-1||y^+| + \int_t^{\infty} |\Phi(t;s)||Q(s)|\, ds\\
 &\le
 C_3|y^+|\int_t^{\infty} |\lambda(\tau)|\, d\tau
 + 
 C_3\int_t^{\infty} |Q(\tau)|\, d\tau,
\end{align*}
as desired.
\end{proof}

\noindent{\em Proof of Proposition~\ref{prp_asym_prof}.}\ 
We first show the pointwise convergence of $V(t;\sigma,\omega)$ 
as $t\to +\infty$. 
We note that \eqref{finite_propa} implies $V(t;\sigma,\omega)=0$ if 
$\sigma\ge R$. In what follows, 
we fix $(\sigma,\omega)\in  (-\infty,R]\times \Sph^1$ and 
introduce 
\[
 \rho(t)=\rho(t;\sigma,\omega)
 :=
 V_1(t;\sigma,\omega)K_1(t;\sigma,\omega)
 -
 V_2(t;\sigma,\omega)K_2(t;\sigma,\omega)
\]
so that 
\[
 \frac{1}{2} \pa_ t \Bigl((V_1(t))^2- (V_2(t))^2\Bigr)
 =V_1(t)\pa_tV_1(t) -V_2(t)\pa_tV_2(t) 
= \rho(t).
\]
It follows from \eqref{t_0}, \eqref{V} and \eqref{K} that 
\begin{align*}
 \int _{t_{0,\sigma}}^{\infty}\left| \rho(\tau;\sigma,\omega)\right|\, d\tau 
 &\le 
 \int _{t_{0,\sigma}}^{\infty} 
 \sum^2_{j=1} |V_j(\tau;\sigma,\omega)K_j(\tau;\sigma,\omega)|\, 
 d\tau \\
 &\le 
 \int _{t_{0,\sigma}}^{\infty} 
 C \eps^2 \jb{\sigma}^{-3/2}\tau^{2\mu-3/2}\, 
 d\tau \\
 &\le 
 C\eps^2 \jb{\sigma}^{-3/2}(t_{0,\sigma})^{2\mu-1/2}\\
 &\le
 C\eps^2 \jb{\sigma}^{2\mu-2}.
\end{align*}
Therefore we obtain 
\begin{align}\label{1--2}
\left(V_1(t;\sigma,\omega)\right)^2- \left(V_2(t;\sigma,\omega)\right)^2 
&=
\left(V_1(t_{0,\sigma};\sigma,\omega)\right)^2- \left(V_2(t_{0,\sigma};\sigma,\omega)\right)^2
+
2 \int _{t_{0,\sigma}}^{t} \rho(\tau;\sigma,\omega)\, d\tau 
 \notag\\
&=
m(\sigma,\omega)-r(t;\sigma,\omega)
\end{align}
for $t\geq t_{0,\sigma}$, where
\begin{align*}
m=m(\sigma,\omega)
:=
\left( V_1(t_{0,\sigma};\sigma,\omega) \right) ^2
- \left( V_2(t_{0,\sigma};\sigma,\omega)\right)^2 
+2\int_{t_{0,\sigma}}^{\infty}  \rho(\tau;\sigma,\omega)\, d\tau
\end{align*}
and
\begin{align*}
r(t)=r(t;\sigma,\omega):=2\int _{t}^{\infty}\rho(\tau;\sigma,\omega)\, d\tau.
\end{align*}
Note that
\[
 |m|
 \le 
 |V(t_{0,\sigma})|^2
 +C\int _{t_{0,\sigma}}^{\infty}|\rho(\tau)|\, d\tau 
\le 
 C\eps^2 \jb{\sigma}^{2\mu-2} \label{m}
\]
and
\begin{align}
 |r(t)|\le 
C\int _{t}^{\infty} |\rho(\tau)|\, d\tau
\le C\eps^2 \jb{\sigma}^{-3/2}t^{2\mu - 1/2}. 
\label{R(t)}
\end{align}
Now we divide the argument into three cases according to the sign of 
$m(\sigma,\omega)$ as follows.

\begin{itemize}
\item
\underline{\bf Case 1: $m(\sigma,\omega)>0$.} \ 
First we focus on the asymptotics  for $V_2(t)$. 
By \eqref{profileV}, \eqref{V}, \eqref{K}, \eqref{1--2} and \eqref{R(t)}, 
we have 
\begin{align*}
 \pa_t V_2(t)
&=
\frac{-1}{2t} V_2(t)^3 -\frac{m}{2t} V_2(t) + \frac{r(t)}{2t} V_2(t) 
 + K_2 (t)\\
&\le 
\frac{-1}{2t} V_2(t)^3 -\frac{m}{2t} V_2(t) 
 +C\eps \jb{\sigma}^{-\mu-1/2}t^{2\mu-3/2},
\end{align*}
whence
\begin{align*}
\frac{1}{2}\pa_t \left( t^m V_2(t)^2 \right)
&=
 t^mV_2(t)\left(\pa_t V_2(t) +\frac{m}{2t}V_2(t)\right) \\
&\le 
 t^m\left(\frac{-1}{2t} V_2(t)^4 
 + C\eps^2 \jb{\sigma}^{-3/2}t^{2\mu-3/2}\right) \\
&\le 
 C\eps^2 \jb{\sigma}^{-3/2}t^{2\mu +m- 3/2}.
\end{align*}
Integration in $t$ leads to
\begin{align*}
t^mV_2(t)^2 -(t_{0,\sigma})^m V_2(t_{0,\sigma})^2
&\le
C\eps^2 \jb{\sigma}^{-3/2}\int^t_{t_{0,\sigma}}\tau^{2\mu +m- 3/2}d\tau \\
&\le 
C\eps^2 \jb{\sigma}^{-3/2}(t_{0,\sigma})^{2\mu +m- 1/2} \\
&\le C\eps^2 \jb{\sigma}^{2\mu+m-2}
\end{align*}
for $t\ge t_{0,\sigma}$. Therefore we deduce that 
\begin{align}
 |V_2(t)| \le C\eps \jb{\sigma}^{\mu +m/2-1} t^{-m/2}.
\label{decay_V2}
\end{align}
In particular, $V_2(t)\to 0$ as $t\to +\infty$. 
Next we turn our attentions to the asymptotics for $V_1(t)$. 
Since $V_1(t)$ solves $V_1'(t)=\lambda(t) V_1(t) +Q(t)$ with 
$\lambda(t)=-{V_2(t)^2}/{t}$ and $Q(t)=K_1(t:\sigma,\omega)$,
we can apply Lemma~\ref{lem_ODE} to $V_1(t)$. Then we have
\[
 |V_1(t)-W_1^+|
 \le
 C\int_t^{\infty} \left(\frac{|W_1^+| |V_2(\tau)|^2}{\tau}
 + 
 |K_1(\tau)|\right)\, d\tau,
\]
where
\begin{align*}
W_1^{+}
&=W_1^{+}(\sigma,\omega)\\
&:=
V_1(t_{0,\sigma};\sigma,\omega)
e^{-\int_{t_{0,\sigma}}^{\infty} V_2(\tau;\sigma,\omega)^2\frac{d\tau}{\tau}}
+
\int_{t_{0,\sigma}}^{\infty} 
K_1(s;\sigma,\omega)
e^{-\int_s^{\infty} V_2(\tau;\sigma,\omega)^2\frac{d\tau}{\tau}}\, ds.
\end{align*}
By \eqref{V}, \eqref{K} and \eqref{decay_V2}, we have
\begin{align}\label{W}
|W_1^{+}|
&\le 
|V_1(t_{0,\sigma})| 
+ \int_{t_{0,\sigma}}^{\infty} |K_1(s)| ds 
\le 
 C\eps \jb{\sigma}^{\mu -1}
\end{align}
and
\begin{align*}
\int_t^{\infty} \left(\frac{|W_1^+| |V_2(\tau)|^2}{\tau}
 + 
 |K_1(\tau)|\right)\, d\tau
&\le 
C\int_t^{\infty} 
\left(
 \frac{\eps^3 \jb{\sigma}^{3\mu+m-3}}{\tau^{1+m}}
 +
 \frac{\eps \jb{\sigma}^{-\mu-1/2}}{\tau^{3/2-2\mu}}
\right)\, d\tau\\
&\le
\frac{C\eps^3 \jb{\sigma}^{3\mu+m-3}}{mt^{m}}
+
\frac{C\eps \jb{\sigma}^{-\mu-1/2}}{t^{1/2-2\mu}}.
\end{align*}
Therefore we conclude that $V_1(t) \to W_1^{+}$ as $t\to +\infty$.

\item
\underline{\bf Case 2: $m(\sigma,\omega)<0$.}\ 
Similarly to the previous case, we have
\begin{align*}
\lim _{t\to \infty}|V_1(t;\sigma,\omega)|=0,
\qquad 
\lim _{t\to \infty}|V_2(t;\sigma,\omega)-W_2^{+}(\sigma,\omega)|=0, 
\end{align*}
where 
\begin{align*}
W_2^{+}(\sigma,\omega)
:=
V_2(t_{0,\sigma};\sigma,\omega)
e^{-\int_{t_{0,\sigma}}^{\infty} V_1(\tau;\sigma,\omega)^2\frac{d\tau}{\tau}}
+
\int_{t_{0,\sigma}}^{\infty} 
K_2(s;\sigma,\omega)
e^{-\int_{s}^{\infty} V_1(\tau;\sigma,\omega)^2\frac{d\tau}{\tau}}\, ds.
\end{align*}
Remark that $|W_2^+|\le C\eps \jb{\sigma}^{\mu-1}$.

\item
\underline{\bf Case 3: $m(\sigma,\omega)=0$.}\ 
By \eqref{profileV}, \eqref{V}, \eqref{K}, \eqref{1--2} and \eqref{R(t)}, 
we have
\begin{align*}
\pa_t \left( V_1(t)^2\right)
&=\frac{-1}{t} V_1(t)^4 - \frac{r(t)}{t} V_1(t)^2 + 2V_1(t)K_1(t) \\
&\le \frac{-1}{t} (V_1(t)) ^4 + C\eps^2 \jb{\sigma}^{-3/2}t^{2\mu - 3/2}
\end{align*}
for $t \ge t_{0,\sigma}$. 
Thus we can apply Lemma~\ref{lem_M} with $\Phi(t)=V_1(t)^2$ to obtain
\begin{align*}
|V_1(t)|\le \frac{C}{\sqrt{\log t}}\to 0 \qquad (t\to +\infty).
\end{align*}
Also \eqref{1--2} gives us $|V_2(t)|=\sqrt{V_1(t)^2 +r(t)} \to 0$ 
as $t\to \infty$.
\end{itemize}

Summing up the three cases above, we deduce that 
$V(t;\sigma,\omega)$ converges as $t\to +\infty$ 
for each fixed $(\sigma,\omega)\in \R\times\Sph^1$.
In order to show \eqref{L2convergence}, we set
\begin{align*}
&V^{+}_1(\sigma,\omega):=
\left\{\begin{array}{cl}
  W^{+}_1(\sigma,\omega) & \left(m(\sigma,\omega)>0 \right),\\
  0 & \left(m(\sigma,\omega)\le 0 \right),
\end{array}\right.
\\[3mm]
&V^{+}_2(\sigma,\omega):=
\left\{ \begin{array}{cl}
 0 & \left( m(\sigma,\omega) \ge0  \right),\\
 W^{+}_2(\sigma,\omega) & \left( m(\sigma,\omega) < 0 \right),
\end{array}\right.
\end{align*}
and $V^{+}(\sigma,\omega)=(V^{+}_j(\sigma,\omega))_{j=1,2}$ 
for $(\sigma,\omega) \in \R\times\Sph^1$. 
Then, by virtue of \eqref{W}, we have $V^{+}\in L^2(\R\times\Sph^1)$ 
and
\begin{align*}
\left| \chi_t(\sigma)V(t;\sigma,\omega)-V^{+}(\sigma,\omega) \right|^2
\le 
C\eps^2 \jb{\sigma}^{2\mu-2}\in L^1(\R\times\Sph^1)
\end{align*}
for all $t\geq t_{0,\sigma}$. Moreover, it holds that
\[
 \lim_{t\to\infty}
 \left| \chi_t(\sigma)V(t;\sigma,\omega)-V^{+}(\sigma,\omega) \right|^2
 =0
\]
for each fixed $(\sigma,\omega)\in\R\times\Sph^1$. 
Consequently,  Lebesgue's dominated convergence theorem yields 
\eqref{L2convergence}.

\section{Proof of Theorem~\ref{thm_main}}  \label{sec_proof_main}

We are going to prove Theorem~\ref{thm_main}. First we recall the following 
useful lemma.

\begin{lem}[\cite{K1} Theorem 2.1]\label{lem_K}
For $\phi \in C\left( [0,\infty);\dot H^1(\R^2) \right) 
\cap C^1\left( [0,\infty); L^2(\R^2) \right)$, 
the following two assertions $\mathrm {(i)}$ and $\mathrm {(ii)}$ are 
equivalent:
\begin{itemize}	
\item[$\mathrm {(i)}$] 
There exists $(\phi^{+}_0,\phi^{+}_1) \in \dot H^1(\R^2) \times L^2(\R^2)$ 
such that
\begin{align*}
 \lim_{t \to \infty} \| \phi(t)-\phi^{+}(t) \|_E=0,
\end{align*}
where $\phi^{+}\in C\left( [0,\infty);\dot H^1(\R^2) \right) 
\cap C^1\left( [0,\infty);L^2(\R^2) \right)$ is a unique solution to 
$\Box  \phi^{+}=0$, $\phi^{+}(0)=\phi_0^{+}$, $\pa\phi^{+}(0)=\phi_1^{+}$.
\item[$\mathrm {(ii)}$] 
There exists $\Phi=\Phi(\sigma,\omega) \in L^2(\R \times \Sph^1)$ 
such that
\begin{align*}
\lim _{t \to \infty} 
\| \pa \phi (t,\cdot) - \hat \omega(\cdot)\Phi^{\sharp}(t,\cdot) \|_{L^2(\R^2)}=0,
\end{align*}
where $\hat{\omega}(x)=(-1,x_1/|x|, x_2/|x|)$ and 
$\Phi^{\sharp}(t,x)=|x|^{-1/2}\Phi(|x|-t,x/|x|)$. 
\end{itemize}
\end{lem}

By virtue of this lemma, 
to prove that $u_1$ is asymptotically free, it is sufficient to 
show 
\begin{align}\label{key}
\lim _{t \to \infty} 
\| \pa u_1 (t,\cdot) - \hat \omega(\cdot)V^{+,\sharp}_1(t,\cdot) \|_{L^2(\R^2)}
=0
\end{align}
for $V_1^{+}(\sigma,\omega)$ obtained in Section \ref{sec_asympt}. 
To prove \eqref{key}, we split 
\begin{align*}
\| 
 \pa &u_1 (t,\cdot) - \hat \omega(\cdot)V^{+,\sharp}_1(t,\cdot) 
\|_{L^2(\R^2)}^2
\\ 
=&
\int_{\R^2}| \pa u_1 (t,x) - \hat \omega(x)|x|^{-1/2}V^{+}_1(|x|-t,x/|x|) |^2dx 
\\
\le& 
2\int_{\R^2 \setminus\Lambda_{\infty}} 
| \pa u_1 (t,x) - \hat \omega(x)|x|^{-1/2}V_1(t;|x|-t,x/|x|) |^2 \, dx 
\\
& +2\int_{\Lambda_{\infty}}
| \pa u_1 (t,x) - \hat \omega(x)|x|^{-1/2}V_1(t;|x|-t,x/|x|) |^2 \, dx
\\
& +2\int_0^{\infty}\int_{\Sph^1} 
 | \hat \omega(r\omega)V_1(t;r-t,\omega) 
   - \hat \omega(r\omega)V_1^{+}(r-t,\omega) |^2 \, 
dS_{\omega}dr \\
=:& J_1(t)+J_2(t)+J_3(t).
\end{align*}
To show the decay for $J_1(t)$, 
we note that 
$\jb{t+|x|} \le C \jb{t-|x|}$ on $\R^2 \setminus \Lambda_{\infty}$. 
Then $(\ref{apriori2})$ and $(\ref{V})$ imply
\begin{align*}
J_1(t)
&\le 
C\eps^2\int_{\R^2 \setminus\Lambda_{\infty}}
\Bigl(\jb{t-|x|}^{-1} \jb{t+|x|}^{2\mu -2}
+|x|^{-1}\jb{t-|x|}^{2\mu-2}\Bigr)\, dx 
\\
&\le 
C\eps^2\int_{\R^2 \setminus\Lambda_{\infty}}|x|^{-1} \jb{t+|x|}^{2\mu -2}\, dx 
\\
&\le 
C\eps^2\int_{0}^{\infty}\int_{\Sph^1} (1+t+r)^{2\mu -2}dS_{\omega}\, dr 
\\
&\le 
C\eps^2 (1+t)^{2\mu-1}.
\end{align*}
As for $J_2(t)$, 
we see from Lemma \ref{lem1} and \eqref{apriori2} that 
\begin{align*}
J_2(t)
&= 
2\int_{\Lambda_{\infty}}|x|^{-1}\left| |x|^{1/2}\pa u_1(t,x)
-\hat \omega(x)\mathcal{D}\left(|x|^{1/2}u_1(t,x)\right)\right|^2 \, dx 
\\
&\le 
C\int_{\Lambda_{\infty}}|x|^{-1}\jb{t+|x|}^{-1}|u(t,x)|_1^2 
\, dx 
\\
&\le 
C\eps^2\int_{\R^2}|x|^{-1}\jb{t+|x|}^{2\mu-2}\, dx 
\\
&\le 
C\eps^2 (1+t)^{2\mu-1}.
\end{align*}
Finally, it follows from \eqref{L2convergence} that
\begin{align*}
J_3(t)
&\le 
C\int_0^{\infty}\int_{\Sph^1} 
\left| V_1(t;r-t,\omega) - V_1^{+}(r-t,\omega) \right|^2
\, dS_{\omega}dr 
\\
&\le 
C\int_{-t}^{\infty}\int_{\Sph^1} 
\left| V_1(t;\sigma,\omega) - V_1^{+}(\sigma,\omega) \right|^2
\, dS_{\omega}d\sigma
 \\
&\le 
C\int_{\R}\int_{\Sph^1}\left|\chi_t(\sigma)V_1(t;\sigma,\omega) 
- V_1^{+}(\sigma,\omega) \right|^2
\, dS_{\omega}d\sigma\\
& \to 0
\end{align*}
as $t\to\infty$. 
Piecing them together, we arrive at \eqref{key}. Similarly we have
\begin{align*}
\lim _{t \to \infty} 
\| 
\pa u_2 (t,\cdot) - \hat \omega(\cdot)V^{+,\sharp}_2(t,\cdot) 
\|_{L^2(\R^2)}=0,
\end{align*}
where $V_2^+$ is from Proposition~\ref{prp_asym_prof}. 
With the aid of Lemma~\ref{lem_K}, 
we conclude that $u_2$ is also asymptotically free.

\medskip
\subsection*{Acknowledgments}
The authors would like to thank Professor Soichiro Katayama, Dr.Yuji Sagawa 
and Daisuke Sakoda for their useful conversations on this subject. 
The work of H.~S. is supported by Grant-in-Aid for Scientific Research (C) 
(No.~17K05322), JSPS.


\end{document}